\documentclass[12pt]{amsart}

\usepackage{amssymb,latexsym}

\usepackage{enumerate}

\makeatletter

\@namedef{subjclassname@2010}{

  \textup{2010} Mathematics Subject Classification}

\makeatother
\newtheorem{thm}{Theorem}[section]

\newtheorem{cor}[thm]{Corollary}
\newtheorem{lem}[thm]{Lemma}
\newtheorem{pro}[thm]{Proposition}
\theoremstyle{definition}

\numberwithin{equation}{section}

\newcommand{\X}{\mathbb{X}}

\newcommand{\ex}{\mathbb{E}}
\newcommand{\re}{\textup{Re}}
\newcommand{\im}{\textup{Im}}
\newcommand{\pr}{\mathbb{P}}
\newcommand{\h}{\mathcal{H}_k}

\newcommand{\sums}{\sideset{}{^\flat}\sum}
\newcommand{\sumh}{\sideset{}{^h}\sum}

\begin{document}

\baselineskip=17pt

\title{Large sums of Hecke eigenvalues of holomorphic cusp forms}

\author{Youness Lamzouri}

\address{Department of Mathematics and Statistics,
York University,
4700 Keele Street,
Toronto, ON,
M3J1P3
Canada}

\email{lamzouri@mathstat.yorku.ca}

\date{}

\begin{abstract}  
Let $f$ be a Hecke cusp form of weight $k$ for the full modular group, and let $\{\lambda_f(n)\}_{n\geq 1}$ be the sequence of its normalized Fourier coefficients. Motivated by the problem of the first sign change of $\lambda_f(n)$, we investigate the range of $x$ (in terms of $k$) for which there are cancellations in the sum $S_f(x)=\sum_{n\leq x} \lambda_f(n)$. We first show that  $S_f(x)=o(x\log x)$ implies that $\lambda_f(n)<0$ for some $n\leq x$. We also prove that $S_f(x)=o(x\log x)$ in the range $\log x/\log\log k\to \infty$ assuming the Riemann hypothesis for $L(s, f)$, and furthermore that this range is best possible unconditionally. More precisely, we establish the existence of many Hecke cusp forms $f$ of large weight $k$, for which $S_f(x)\gg_A x\log x$, when $x=(\log k)^A.$  Our results are $GL_2$ analogues of  work of Granville and Soundararajan for character sums, and could also be generalized to other families of automorphic forms. 
\end{abstract}


\thanks{The author is partially supported by a Discovery Grant from the Natural Sciences and Engineering Research Council of Canada.}

\maketitle

\section{Introduction}
Let $k$ be a positive even integer, and denote by $\h$ the set of Hecke cusp forms of weight $k$ for the full modular group $\Gamma=SL(2, \mathbb{Z})$. Then, $\h$ is an orthonormal basis for the space of holomorphic cusp forms of weight $k$ for $\Gamma$ and we have 
$$ |\h|=\frac{k}{12}+O\left(k^{2/3}\right).$$
Given $f\in \h$, its Fourier expansion can be written in the form
$$ f(z)=\sum_{n=1}^{\infty}\lambda_f(n) n^{(k-1)/2}e(nz), \ \  \text{ for } \ \ \im(z)>0,$$
where $e(z)=e^{2\pi iz}.$
The $\lambda_f(n)$ are the normalized
eigenvalues of the Hecke operators $T_n$, and satisfy the well-known Hecke relations:
\begin{equation}\label{Hecke}
 \lambda_f(m)\lambda_f(n)=\sum_{d|(m,n)}\lambda_f\left(\frac{mn}{d^2}\right),
\end{equation}
for all $m, n\geq 1$. In particular, $\lambda_f$ is a real-valued multiplicative function of $n$. Moreover, it also satisfies the following deep bound due to Deligne
\begin{equation}\label{Deligne}
\left|\lambda_f(n)\right|\leq \tau(n),
\end{equation}
where $\tau$ is the divisor function. These facts are standard and may be found for example  in Chapter 14 of \cite{IwKo}. 

In \cite{KLSW}, Kowalski, Lau, Soundararajan and Wu studied the signs of the sequence $\lambda_f(n)$.
Their results show a strong analogy between these signs and the values of quadratic Dirichlet characters, and especially between the first negative Fourier coefficient and the problem of the least quadratic non-residue, which has a long history in analytic number theory. 
Let $n_f$ be the smallest positive integer $n$ such that $\lambda_f(n)<0$. The best known bound for $n_f$ is due to Matom\"aki \cite{Ma}, who improved the authors of \cite{KLSW} by showing that 
$$n_f\ll k^{3/4}.$$ This is probably far from the truth, since it is known that $n_f\ll (\log k)^2$ under the assumption of the generalized Riemann hypothesis (GRH). In the other direction, Theorem 3 of \cite{KLSW} shows that $n_f\gg \sqrt{\log k}$ for many Hecke cusp forms $f$ of weight $k$. A folklore conjecture asserts that the correct order of magnitude for the maximal values of $n_f$ should be $(\log k)^{1+o(1)}$.

In this paper, we explore $GL_2$ analogues of certain classical problems concerning short character sums and the least quadratic non-residue. More precisely, we investigate the size of the short sum of Hecke eigenvalues
$$S_f(x):=\sum_{n\leq x} \lambda_f(n),$$
and its relation to the first negative Fourier coefficient of $f$. Our results are inspired by the work of Granville and Soundararajan \cite{GrSo} on  character sums.  In particular, Corollaries \ref{Conditional2} and \ref{Omega2} below can be regarded as $GL_2$ analogues of Corollary A of \cite{GrSo}.

Using Deligne's bound \eqref{Deligne}, one obtains the ``trivial'' bound 
$$
\big| S_f(x)\big|\leq \sum_{n\leq x}\tau(n)=(1+o(1))x\log x.
$$
 Our first result shows that if $S_f(x)$ is substantially smaller than this bound, namely that \begin{equation}\label{SMALL}
S_f(x)=o(x\log x) \ \ \ \ \ \  (\text{as } x, k\to \infty),
\end{equation}
 then we must have $n_f\leq x$. The proof relies on an argument of Kowalski, Lau, Soundararajan and Wu \cite{KLSW}, together with a nice result of Hildebrand \cite{Hi} concerning quantitative lower bounds for mean values of non-negative multiplicative functions. 
 \begin{thm}\label{Cancellations}
Let $f\in \h$. Let $x\geq 2$ and assume that $\lambda_f(n)\geq 0$ for all $n\leq x$. Then, we have
$$ \sum_{n\leq x}\lambda_f(n)\geq c_0 x\log x,$$
for some absolute constant $c_0>0$. 
\end{thm}

Let $f\in \h$. The $L$-function attached to $f$ is defined by 
\begin{equation}\label{LFUNCTION}
L(s, f)=\sum_{n=1}^{\infty} \frac{\lambda_f(n)}{n^s}= \prod_{p} \left(1-\frac{e^{i\theta_{f}(p)}}{p}\right)^{-1} \prod_{p} \left(1-\frac{e^{-i\theta_{f}(p)}}{p}\right)^{-1}, \text{ for } \re(s)>1,
\end{equation}
where $\theta_f(p) \in [0, \pi]$. 
 It is known that $L(s, f)$ extends analytically to the entire complex plane, and satisfies a functional equation that relates $L(s, f)$ to $L(1-s, f)$ (see for example Section 5.11 of \cite{IwKo}). A standard application of Perron's formula together with the convexity bound for $L(s, f)$ imply that
\begin{equation} \label{Convexity}
S_f(x)\ll x^{1/2+\varepsilon} \cdot k^{1/2+\varepsilon}, 
\end{equation}
and hence one has $S_f(x)=o(x\log x)$ in the range $x\geq k^{1+\varepsilon}.$ This range can be improved to $x\geq k^{1-\delta}$, for some $\delta>0$, by using subconvexity bounds for $L(s, f)$ (see for example \cite{MiVe}). Furthermore, assuming GRH for $L(s, f)$ one has the much stronger bound
\begin{equation}\label{WeakRangeGRH}S_f(x)\ll x^{1/2+\varepsilon} \exp\left(c_1\frac{\log k}{\log\log k}\right),
\end{equation}
for some absolute constant $c_1>0$. This shows that \eqref{SMALL} is valid in the larger range
$$x\geq \exp\left(c_2\frac{\log k}{\log\log k}\right)
$$
for some constant $c_2>0$, conditionally on the GRH.
Exploiting an idea of Montgomery and Vaughan \cite{MoVa}, we substantially improve this range under the assumption of GRH.
\begin{cor}\label{Conditional2}
Let $f\in \h$, and assume GRH for $L(s, f)$. In the range $\log x/\log\log k\to \infty$, we have
$$ \sum_{n\leq x} \lambda_f(n)=o(x\log x).$$
\end{cor}
  We shall deduce this result from the following theorem,  which shows that under GRH, we can approximate $S_f(x)$ by the corresponding sum of $\lambda_f(n)$ over \emph{friable} (or smooth) numbers $n$, which are positive integers having only small prime factors. A positive integer $n$ is said to be $y$-friable if $P(n)\leq y$, where $P(n)$ denotes the largest prime factor of $n$, with the standard convention $P(1)=1$.
 \begin{thm}\label{Conditional}
Let $f\in \h$, and assume GRH for $L(s, f)$. Then, for all real numbers $x, y$ such that $(\log k)^2(\log\log k)^8\leq y\leq x\leq k$ we have
$$ \sum_{n\leq x} \lambda_f(n)=\sum_{\substack{n\leq x\\ P(n)\leq y}}\lambda_f(n)+O\left(\frac{(\log k)(\log y)^4}{\sqrt{y}}x\log x\right).$$
\end{thm} 
For an arithmetic function $g$, we define
$$ \Psi(x, y;g):=\sum_{\substack{n\leq x\\ P(n)\leq y}} g(n).$$
The asymptotic behaviour of $\Psi(x, y; g)$ was investigated for a large class of multiplicative functions $g$ by several authors, and notably by Tenenbaum and Wu \cite{TeWu1}.
When $g$ is  the divisor function $\tau$,
de Bruijn and van Lint  \cite{BrLi} proved that there exists a differentiable function $\rho_2: [0, \infty)\to \mathbb{R}$ such that
\begin{equation}\label{Bruijn}
\Psi(x, y;\tau)\sim \rho_2(u) \cdot x\log y,  \text{ where } u:=\frac{\log x}{\log y},
\end{equation}
in the range $u\ll 1$. The function $\rho_2$ is defined by the differential-difference equation
\begin{equation}\label{DiffDiff}
u\rho_2'(u)=\rho_2(u)-2\rho_2(u-1),
\end{equation}
subject to the initial condition $\rho_2(u)=u$ for $0\leq u\leq 1$. It is known that $\rho_2(u)>0$ for any $u>0$ and that $\rho_2(u)=u^{-u(1+o(1))}$ for large $u$ (see for example \cite{He}). In fact, $\rho_2$ is the square convolution of the standard Dickman-de Bruijn function $\rho$, which appears in the asymptotic formula for the counting function of friable integers.
The range of validity of the asymptotic formula \eqref{Bruijn} was improved to $u\leq \exp\left((\log y)^{3/5-\varepsilon}\right)$ by Smida \cite{Sm}, and hence in this range we have
$$ \Psi(x, y; \lambda_f)\leq \Psi(x, y;\tau) \ll u^{-u(1+o(1))} x\log x,$$ 
by \eqref{Deligne}. For our purposes, it is enough to use the following weaker bound that holds uniformly for $10\leq y\leq x$ (see Lemma \ref{Friable} below)
\begin{equation}\label{DecayFriable}
\Psi(x, y;\tau) \ll e^{-u/2} x\log x.
\end{equation}
Combining this bound with Theorem \ref{Conditional} imply Corollary \ref{Conditional2}.

We now investigate the largest range of $x$ (in terms of $k$) for which one has 
\begin{equation}\label{LARGE}
S_f(x)\gg x\log x.
\end{equation}
Recall that $n_f\gg \sqrt{\log k}$ for many Hecke cusp forms $f$ of weight $k$ by Theorem 3 of \cite{KLSW}. In view of Theorem \ref{Cancellations}, this shows that  \eqref{LARGE} is valid for such $f$ with $x=\sqrt{\log k}$. On the other hand, since $n_f\ll (\log k)^2$ on GRH, one might guess that \eqref{LARGE} does not hold in the range $x\gg (\log k)^{2+\varepsilon}$. We prove that this is not the case, by showing that for any $A>1$, there are many Hecke cusp forms $f$ of weight $k$ such that \eqref{LARGE} holds for $x=(\log k)^A$. This shows that the range of Corollary \ref{Conditional2} is best possible, and that conditionally on GRH the converse of Theorem \ref{Cancellations} does not hold.
\begin{cor}\label{Omega2}
Let $k$ be a large even integer. Let $A>1$ be fixed, and $x=(\log k)^A$. There are at least $k^{1-1/\log\log k}$ Hecke cusp forms $f\in \h$ such that 
$$ \sum_{n\leq x} \lambda_f(n)\gg_A x\log x.$$
\end{cor}
We shall deduce this result from the following theorem.
\begin{thm}\label{Omega}
Let $k$ be a large even integer. Let $A>1$ be fixed, $y=\log k/\log\log k$ and $x=(\log k)^A$. There are at least $k^{1-1/\log\log k}$  Hecke cusp forms $f\in \h$ such that 
$$ \left|\sum_{n\leq x} \lambda_f(n)\right| \geq \Psi(x, y; \tau) \left(1+O\left(\frac{1}{\sqrt{\log\log k}}\right)\right).$$ 
\end{thm}

The key idea in the proof of Theorem \ref{Omega} is to compare large moments of $S_f(x)$ (as $f$ varies in $\h$) with those of a corresponding probabilistic random model. 
This model was introduced by Cogdell and Michel in \cite{CoMi} to study the complex moments of symmetric power $L$-functions at $s=1$, and was subsequently used by various authors (see for example \cite{LRW} and \cite{LRW2}) to explore similar problems. 
To describe this probabilistic model we consider the compact group $G=SU(2)$ endowed with its natural Haar measure $\mu_G$; we then let $G^{\natural}$ be the set of conjugacy classes of $G$ endowed with the Sato-Tate measure $\mu_{st}$ (i.e. the direct image of $\mu_G$ by the canonical projection). 
Let $\{g^{\natural}_p\}_{p \text{ prime }}$ be a sequence of independent random variables, with values in $G^{\natural}$ and distributed according to the measure $\mu_{st}$.
We construct the sequence of random variables $\{\X(n)\}_{n\geq 1}$ by first defining 
$$
\X(p^a)=\text{tr}\left(\text{Sym}^a \left(g^{\natural}_p\right)\right)
$$ for a prime $p$ and a positive integer $a$, where $\text{Sym}^a$ is the symmetric $a$-th power representation of the standard representation of $GL_2$. We then extend the $\X(p^a)$ multiplicatively by letting $\X(1)=1$ and 
$$
\X(n)= \X(p_1^{a_1})\cdots \X(p_{\ell})^{a_{\ell}}
$$ if the prime factorization of $n$ is $n= p_1^{a_1}\cdots p_{\ell}^{a_{\ell}}.$
We shall explore this probabilistic model and the motivation behind it in details in Section 3. 
Using the Petersson trace formula (see Lemma \ref{Petersson} below), we show that in a certain range of $x$, large (weighted) moments of $S_f(x)$ are very close to those of the sum of random variables $\sum_{n\leq x} \X(n)$. We then estimate the moments of this sum by first restricting the random variables $\X(n)$ to those indexed by $y$-friable integers $n$, and then controlling these by restricting the  range of the random variables $\X(p)$ for the primes $p\leq y$.

 Our approach is flexible and could be further generalized to obtain similar results for other families of automorphic forms. In particular, our results hold \emph{mutatis mutandis} for primitive Hecke cusp forms of weight $2$ and prime level $q$ (in the level aspect), with the extra condition that $x<q$ in Theorem \ref{Cancellations}. One should also obtain the analogues of Theorems \ref{Conditional} and \ref{Omega} for Fourier coefficients of the symmetric square and other symmetric power $L$-functions attached to primitive Hecke cusp forms, assuming their automorphy. 


\section{The size of $S_f(x)$ and the first negative Hecke eigenvalue: Proof of Theorem \ref{Cancellations}}
Let $p$ be a prime number. It follows from \eqref{LFUNCTION} that $\lambda_f(p)=2\cos\theta_f(p)$ and more generally we have 
$$ \lambda_f(p^b)=\frac{\sin((b+1)\theta_f(p))}{\sin\theta_f(p)},$$
for any integer $b\geq 0$, by the Hecke relations \eqref{Hecke}.

Let $\alpha: [0, 1]\to [-2, 2]$ be defined by $\alpha(0)=2$ and $\alpha(t)=2\cos(\pi/(m+1))$ if $1/(m+1)< t\leq 1/m$, for $m\in \mathbb{N}$. For $x\geq 2$, let $h_x$ be the multiplicative function supported on square-free numbers and defined on the primes by 
$$
h_x(p)=\begin{cases} \alpha\left(\frac{\log p}{\log x}\right) & \text{ if } p\leq x,\\ 0 & \text{ otherwise.}\end{cases}
$$
By exploiting the Hecke relations \eqref{Hecke}, we obtain the following lemma which is essentially proved in  \cite{KLSW}.
\begin{lem}\label{PositiveSigns}
Let $f\in \h$. Let $x\geq 2$ be such that $\lambda_f(n)\geq 0$ for all $n\leq x$. Then, we have
$$
\sum_{n\leq x} \lambda_f(n)\geq \sum_{n\leq x} h_x(n).
$$
\end{lem}
\begin{proof} By our assumption we have
$$
\sum_{n\leq x} \lambda_f(n)\geq \sums_{n\leq x} \lambda_f(n),
$$
where $\sums$ restricts the summation to squarefree integers.  Since $h_x(n)\geq 0$ for all squarefree $n$, it thus  suffices to show that $\lambda_f(p)\geq h_x(p)$ for all primes $p\leq x$. Let $p\leq x$ be a prime number, and $m\geq 1$ be such that $x^{1/(m+1)}< p\leq x^{1/m}$. Then, for all integers $1\leq j\leq m$ we have
$$0\leq \lambda_f(p^j)=\frac{\sin((j+1)\theta_f(p))}{\sin\theta_f(p)}.$$
This implies $0\leq \theta_f(p)\leq \pi/(m+1)$ and hence that
$$\lambda_f(p)=2\cos\theta_f(p)\geq h_x(p),$$
as desired.
\end{proof}
In order to complete the proof of Theorem \ref{Cancellations}, we need to obtain a lower bound for $\sum_{n\leq x} h_x(n)$. We prove the following result.
\begin{pro}\label{LOWERBOUNDH} 
There is an absolute constant $c_0>0$ such that for all large $x$ we have
$$\sum_{n\leq x} h_x(n)\geq c_0x\log x.$$
\end{pro}
Combining this result with Lemma \ref{PositiveSigns} imply Theorem \ref{Cancellations}. In order to prove Proposition \ref{LOWERBOUNDH}, we shall use the following theorem of Hildebrand \cite{Hi} which provides quantitative lower bounds for mean values of certain non-negative multiplicative functions. 
\begin{thm}[Theorem 2 of \cite{Hi}]\label{Hildebrand}
Let $2\leq z\leq x$ be real numbers. Let $g$ be a multiplicative function supported on squarefree numbers, such that $0\leq g(p)\leq K$ for some constant $K\geq1$ and all primes $p$. Then, we have
\begin{align*}
\frac{1}{x}\sum_{n\leq x}g(n )\geq & \frac{e^{-\gamma(K-1)}}{\Gamma(K)}\prod_{p\leq x} \left(1-\frac{1}{p}\right)\left(1+\frac{g(p)}{p}\right)\\
& \times \left\{\sigma\left(\exp\left(\sum_{z\leq p\leq x} \frac{(1-g(p))^{+}}{p}\right)\right)\left(1+O\left(\frac{\log^{\beta} z}{\log^{\beta} x}\right)\right)+O\left(e^{-\left(\frac{\log x}{\log z}\right)^{\beta}}\right)\right\},
\end{align*}
where $\gamma$ is the Euler-Mascheroni constant, $y^+=\max\{y, 0\}$, $\beta>0$ is an absolute constant, and $\sigma(u)$ is a continuously differentiable function of $u\geq 1$ that satisfies $\sigma(u)\gg u^{-u}$. Furthermore, the implicit constants in the $O$-terms depend on $K$ only. 
\end{thm}
We also need the following lemma.
\begin{lem}\label{SumPrimesH}
Let $x$ be large. Then, we have
$$\sum_{p\leq x} \frac{h_x(p)}{p}= 2\log\log x+O(1).$$
\end{lem}
\begin{proof} First, note that
\begin{equation}\label{SUMHX}
\sum_{p\leq x} \frac{h_x(p)}{p} 
= \sum_{1\leq m\leq \frac{\log x}{\log 2}}2\cos\left(\frac{\pi}{m+1}\right) \sum_{x^{1/(m+1)}\leq p<x^{1/m}}\frac{1}{p}.
\end{equation}
Let $M< \sqrt{\log x}$ be a large positive integer to be chosen later. Then, we have
\begin{equation}\label{SUMHX1}
\begin{aligned}
&\sum_{m\leq M} \cos\left(\frac{\pi}{m+1}\right) \sum_{x^{1/(m+1)}\leq p<x^{1/m}}\frac{1}{p}\\
= \ & \sum_{m\leq M}\cos\left(\frac{\pi}{m+1}\right)\left(\log\left(\frac{m+1}{m}\right)+O\left(\frac{1}{\log(x^{1/(m+1)})}\right)\right)\\
= \ & \sum_{m\leq M} \left(\frac{1}{m}+O\left(\frac{1}{m^2}\right)\right) +O\left(\frac{M^2}{\log x}\right) =\log M+ O\left(1\right).
\end{aligned}
\end{equation}
Furthermore, we have
\begin{equation}\label{SUMHX2}
\begin{aligned}
\sum_{M< m\leq \frac{\log x}{\log 2}} \cos\left(\frac{\pi}{m+1}\right) \sum_{x^{1/(m+1)}\leq p<x^{1/m}}\frac{1}{p} 
&= \sum_{M< m\leq \frac{\log x}{\log 2}}\left(1+O\left(\frac{1}{m^2}\right)\right)\sum_{x^{1/(m+1)}\leq p<x^{1/m}}\frac{1}{p}\\
&= \left(1+O\left(\frac{1}{M^2}\right)\right) \sum_{p< x^{1/(M+1)}} \frac{1}{p}\\
& = \log\log x- \log M+ O\left(\frac{\log\log x}{M^2}+1\right).
\end{aligned}
\end{equation}
Choosing $M=[\log\log x]$, and inserting the estimates \eqref{SUMHX1} and \eqref{SUMHX2} in  \eqref{SUMHX} completes the proof.
\end{proof}
\begin{proof}[Proof of Proposition \ref{LOWERBOUNDH}] Note that $h_x(p)< 1$ if and only if $p>x^{1/2}$. Therefore, for all $z\leq x^{1/2}$ we have
$$ 
\sum_{z\leq p\leq x} \frac{(1-h_x(p))^{+}}{p}=\sum_{x^{1/2}<p\leq x} \frac{1}{p}= \log 2+O\left(\frac{1}{\log x}\right).
$$
Thus, choosing $K=z=2$ in Theorem \ref{Hildebrand} we obtain that
$$
 \sum_{n\leq x} h_x(n) \geq \left(e^{-2\gamma} \sigma(2)+o(1)\right)\frac{x}{\log x} \prod_{p\leq x} \left(1+\frac{h_x(p)}{p}\right).
$$
The result follows from Lemma \ref{SumPrimesH}. 
\end{proof}

\section{Large sums of Hecke eigenvalues : proofs of Theorem \ref{Omega} and Corollary \ref{Omega2}}
In order to prove Theorem \ref{Omega}, we shall compute the moments of $S_f(x)$ as $f$ varies in $\h$. When so doing, we shall use the 
\emph{harmonic weights} that arise naturally in the Petersson trace formula (see Lemma \ref{Petersson} below). 
The harmonic weight of $f\in \h$ is defined by
$$ \omega_f= \frac{\Gamma(k-1)}{(4\pi)^{k-1}\langle f,f\rangle}= \frac{2\pi^2}{(k-1) L(1, \text{Sym}^2 f)},$$
where $\langle f,f\rangle$ is the Petersson inner product, and $L(s, \text{Sym}^2 f)$ is the symmetric square $L$-function of $f$. Given a sequence $(\alpha_f)_{f\in \h}$, its harmonic average is defined as the sum
$$ \sumh_{f\in \h}\alpha_f=\sum_{f\in \h} \omega_f \alpha_f,$$
and if $S\subset \h$ we will let $|S|_h$ denote the harmonic measure of $S$, that is
$$ |S|_h:=\sumh_{f\in S}1.$$
Moreover, the classical estimate
\begin{equation}\label{Harmonic1}
 |\h|_h=1+O\left(k^{-5/6}\right),
\end{equation}
 together with the bounds of Goldfeld, Hoffstein and Liemann (see the Appendix of \cite{GHL})
\begin{equation}\label{Harmonic2}
\frac{1}{k \log k}\ll \omega_f\ll \frac{\log k}{k},
\end{equation}
show that the harmonic weight $\omega_f$ is close to the natural weight $1/|\h|$ (since $|\h|\asymp k$), and it defines asymptotically a probability measure on $\h$. 

We shall use the following consequence of the Petersson trace formula which follows from Lemma 2.1 of \cite{RuSo}. 
\begin{lem}\label{Petersson}
Let $k$ be a large even integer, and $n$ be a positive integer such that $n\leq k^2/10^4$. Then, we have
\begin{equation}\label{Petersson2} 
\frac{1}{|\h|_h}\sumh_{f\in \h}\lambda_f(n)=\delta(n)+O\left(k^{-5/6}\right),
\end{equation}
where $\delta(n)=1$ if $n=1$, and is $0$ otherwise.
\end{lem}
\begin{proof}
It follows from Lemma Lemma 2.1 of \cite{RuSo} that
$$\sumh_{f\in \h}\lambda_f(n)=\delta(n)+O\left(e^{-k}\right).$$
The result follows from combining this estimate with \eqref{Harmonic1}.
\end{proof}

The formula \eqref{Petersson2} can be interpreted as follows: Recall that $G^{\natural}$ is the set of conjugacy classes of $G=SU(2)$ endowed with the Sato-Tate measure $\mu_{st}$ (the direct image of the Haar measure $\mu_G$ by the canonical projection). Let $n>1$ and $n=p_1^{a_1}\cdots p_{\ell}^{a_{\ell}}$ be its prime factorization. Then we have the identity
\begin{equation}\label{LambdaTrace}
 \lambda_f(n)= \lambda_f(p_1^{a_1})\cdots  \lambda_f(p_{\ell}^{a_{\ell}})= \text{tr}\left(\text{Sym}^{a_1} \left(g_f(p_1)\right)\right) \cdots \text{tr}\left(\text{Sym}^{a_{\ell}} \left(g_f(p_{\ell})\right)\right),
\end{equation}
where 
$$ g_f(p) =\left(\begin{matrix} e^{i \theta_f(p)}&0\\0&e^{-i \theta_f(p)}\end{matrix}\right).$$ 
Fix now the primes $p_1, \dots, p_{\ell}$. By the identity \eqref{LambdaTrace} together with the Peter-Weyl Theorem and Weyl's equidistribution criterion, the estimate \eqref{Petersson2} applied to integers $n$ divisible only by the primes in $\{p_1, \dots, p_{\ell}\}$ yields the equidistribution of the $\ell$-tuple of conjugacy classes $\{g^{\natural}_f(p_1), \dots, g^{\natural}_f(p_{\ell}) \}_{f\in \h}$ (appropriately weighted by $\omega_f$) into the product of $\ell$ copies of $G^{\natural}$, as $k\to\infty$. Based on this equidistribution result, we construct a probabilistic random model  for the Hecke eigenvalues $\lambda_f(n)$ as follows: let $\{g^{\natural}_p\}_{p \text{ prime }}$ be a sequence of independent random variables, with values in $G^{\natural}$ and distributed according to the measure $\mu_{st}$. We define $\X(1)=1$ and for $n>1$
$$
\X(n)=  \text{tr}\left(\text{Sym}^{a_1} (g^{\natural}_{p_1})\right) \cdots \text{tr}\left(\text{Sym}^{a_{\ell}} (g^{\natural}_{p_{\ell}})\right),
$$
if $n=p_1^{a_1}\cdots p_{\ell}^{a_{\ell}}$ is the prime factorization of $n$. Furthermore, one can easily check that the $\X(n)$ satisfy the Hecke relations \eqref{Hecke}, namely that
$$
 \X(m)\X(n)=\sum_{d|(m, n)}\X\left(\frac{mn}{d^2}\right).
 $$
We prove the following lemma. 
\begin{lem}\label{ExpX} Let $n$ be a positive integer. Then we have
$$ 
\ex(\X(n))= \delta(n).
$$
\end{lem}
\begin{proof} 
Let $n>1$, and write the prime factorization of $n$ as $n=p_1^{a_1}\cdots p_{\ell}^{a_{\ell}}$. First, by the independence of the random variables $g^{\natural}_p$ for different primes $p$, we have
$$ \ex(\X(n))= \prod_{j=1}^{\ell} \ex\left(\text{tr}\left(\text{Sym}^{a_j} (g^{\natural}_{p_j})\right)\right).$$
By Weyl's integration formula, the map
$$ \theta\rightarrow g^{\natural}(\theta)= \left(\begin{matrix} e^{i \theta}&0\\0&e^{-i \theta}\end{matrix}\right)^{\natural},$$
identifies $G^{\natural}$ with the interval $[0,\pi]$ and $\mu_{st}$ with the distribution
$d\mu_{st}(t):=\frac{2}{\pi}\sin^2(t)dt$. Furthermore, note that 
$$\text{Sym}^{a}\left(g^{\natural}(\theta)\right)= \left(\begin{matrix} e^{i a\theta}\\ &e^{i(a-2)\theta}\\ &&\ddots\\ &&& e^{-ia\theta}\end{matrix}\right)^{\natural},
$$
and hence
$$ \text{tr}\left(\text{Sym}^{a} \left(g^{\natural}(\theta)\right)\right)=  \sum_{j=0}^a e^{i (a-2j)\theta}= \frac{\sin((a+1)\theta)}{\sin \theta}.
$$Therefore, for a prime $p$ and a positive integer $a$ we obtain
$$ \ex(\X(p^a))= \frac{2}{\pi} \int_{0}^{\pi}  \frac{\sin((a+1)\theta)}{\sin \theta} \sin^2\theta d\theta=0,$$
since the functions $\{Y_n\}_{n\geq 0}$, defined by
$$Y_n(t):=\frac{\sin((n+1)t)}{\sin t}$$
form an orthonormal basis of $L^2([0,\pi],d\mu_{st})$. This completes the proof.\end{proof}

Using Lemmas \ref{Petersson} and \ref{ExpX} we prove that in a certain range of $x$, the harmonic moments of $S_f(x)$ (as $f$ varies in $\h$) are very close to the moments of the sum of random variables $\sum_{n\leq x} \X(n)$.
\begin{pro}\label{Moments}
Let $k$ be a large even integer. Let $x\geq 2$ and $\ell$ be a positive integer such that $x^{6\ell}\leq k.$ Then, we have
$$ \frac{1}{|\h|_h}\sumh_{f\in \h}\left|\sum_{n\leq x} \lambda_f(n)\right|^{2\ell}= \ex\left(\left|\sum_{n\leq x}\X(n)\right|^{2\ell}\right) +O\left(k^{-1/3}\right).
$$
\end{pro}
In order to prove this proposition, we need to understand the combinatorics of the Hecke relations \eqref{Hecke}. These relations can be written as 
$$ \lambda_f(n_1)\lambda_f(n_2)=\sum_{m|n_1n_2} b_m(n_1, n_2)\lambda_f\left(m\right),$$
where $b_m(n_1, n_2)= 1$ if $m=n_1n_2/d^2$ for some $d|(n_1, n_2)$, and equals $0$ otherwise. More generally, one can write 
\begin{equation}\label{Hecke2}
\lambda_f(n_1)\cdots \lambda_f(n_r)=\sum_{m|\prod_{j=1}^rn_j}b_m(n_1,\dots,n_r)\lambda_f(m),
 \end{equation}
 for some integers $b_m(n_1,\dots,n_r)$. These coefficients have a nice interpretation in terms of the representation theory of $G=SU(2)$. The irreducible characters of $G$ are 
 $$g \to \text{tr}\left(\text{Sym}^{a} (g)\right),$$
 for $a\geq 0$.
 Hence, for $n=p_1^{a_1}\cdots p_{\ell}^{a_{\ell}}$, the character
 $$\chi_n(g_{p_1}, \dots, g_{p_{\ell}})= \text{tr}\left(\text{Sym}^{a_1} \left(g_{p_1}\right)\right)\cdots\text{tr}\left(\text{Sym}^{a_{\ell}} \left(g_{p_{\ell}}\right)\right)
 $$
  is an irreducible character of the product of $\ell$ copies of $G$, and the formula
 $$ \chi_{n_1}\cdots \chi_{n_r}= \sum_{m|\prod_{j=1}^rn_j}b_m(n_1,\dots,n_r) \chi_m$$
 is the decomposition formula for the product of the $r$
characters $\chi_{n_1},\dots, \chi_{n_r}$ in terms of the irreducibles $\chi_m$. In particular, the coefficients $b_m(n_1,\dots,n_r)$ are non-negative, and we also have 
\begin{equation}\label{RandomHECKE}
\X(n_1)\cdots \X(n_r)= \sum_{m|\prod_{j=1}^rn_j}b_m(n_1,\dots,n_r)\X(m).
\end{equation}
Moreover, one can easily prove (either by induction on $r$ or by exploiting the representation theory of $SU(2)$)  that
\begin{equation}\label{BoundBM}
b_m(n_1,\dots, n_r)\leq \tau(n_1)\cdots \tau(n_r).
\end{equation}


\begin{lem}\label{ExpectationRandom}
Let $g$ be a real-valued arithmetic function. For all $x\geq 2$ and positive integers $\ell$ we have
$$ \ex\left(\left|\sum_{n\leq x}\X(n)g(n)\right|^{2\ell}\right)= \sum_{n_1, \dots, n_{2\ell}\le x} b_1(n_1, \dots, n_{2\ell})g(n_1)g(n_2)\cdots g(n_{2\ell}).$$
\end{lem}
\begin{proof}
We have 
\begin{align*}
\ex\left(\left|\sum_{n\leq x}\X(n)g(n)\right|^{2\ell}\right)
&= \ex\left( \sum_{n_1, \dots, n_{2\ell}\le x} \X(n_1)\X(n_2) \cdots \X(n_{2\ell})g(n_1)g(n_2)\cdots g(n_{2\ell})\right)\\
&= \sum_{n_1, \dots, n_{2\ell}\le x}g(n_1)g(n_2)\cdots g(n_{2\ell}) \ex\big(\X(n_1)\X(n_2) \cdots \X(n_{2\ell})\big).
\end{align*}
Moreover, it follows from \eqref{RandomHECKE} that 
$$ \ex\big(\X(n_1)\X(n_2) \cdots \X(n_{2\ell})\big)= \sum_{m | n_1n_2\cdots n_{2\ell}}b_m(n_1, n_2, \dots, n_{2\ell}) \ex(\X(m))=b_1(n_1, n_2, \dots, n_{2\ell}),$$
by Lemma \ref{ExpX}. This completes the proof.
\end{proof}
We deduce the following corollary.
\begin{cor}\label{CompareArithmetic}
Let $g$ and $h$ be arithmetic functions such that $g(n)\geq h(n)\geq 0$ for all $n\geq 1$. Then we have
$$ 
\ex\left(\left|\sum_{n\leq x}\X(n)g(n)\right|^{2\ell}\right)\geq  \ex\left(\left|\sum_{n\leq x}\X(n)h(n)\right|^{2\ell}\right).
$$
\end{cor}

We are now ready to prove Proposition \ref{Moments}. 

\begin{proof}[Proof of Proposition \ref{Moments}] By \eqref{Hecke2} we obtain\begin{align*}
\frac{1}{|\h|_h}\sumh_{f\in \h}\left|\sum_{n\leq x} \lambda_f(n)\right|^{2\ell}
&=\frac{1}{|\h|_h}\sumh_{f\in \h}\sum_{n_1, \dots, n_{2\ell}\le x} \lambda_f(n_1) \cdots \lambda_f(n_{2\ell})\\
&= \sum_{n_1, \dots, n_{2\ell}\le x} \sum_{m | n_1\cdots n_{2\ell}} b_m(n_1, \dots, n_{2\ell}) \frac{1}{|\h|_h}\sumh_{f\in \h} \lambda_f(m).
\end{align*}
Therefore, by Lemma \ref{Petersson} we get
\begin{align*}
\frac{1}{|\h|_h}\sumh_{f\in \h}\left|\sum_{n\leq x} \lambda_f(n)\right|^{2\ell}= & 
\sum_{n_1, \dots, n_{2\ell}\le x}  b_1(n_1, \dots, n_{2\ell})\\
& \ \ \  +O\left(k^{-5/6}\sum_{n_1, \dots, n_{2\ell}\le x} \sum_{m | n_1\cdots n_{2\ell}}b_m(n_1, \dots, n_{2\ell})\right).
\end{align*}
Now, using that $b_m(n_1, \dots, n_{2\ell})\leq \tau(n_1)\cdots \tau(n_{2\ell})$ we deduce that the error term above is 
$$ \ll k^{-5/6}\sum_{n_1, \dots, n_{2\ell}\le x} \tau(n_1)\cdots \tau(n_{2\ell}) \tau(n_1 \cdots n_{2\ell})\ll_{\varepsilon} x^{2\ell\varepsilon} k^{-5/6} \left(\sum_{n\leq x} \tau(n)\right)^{2\ell} \ll k^{-1/3}.$$
using the bound $\tau(n_1\cdots n_{2\ell})\ll_{\varepsilon} (n_1\cdots n_{2\ell})^{\varepsilon}\leq x^{2\ell\varepsilon}$ together with the estimate $\sum_{n\leq x} \tau(n)\ll x\log x$. Appealing to Lemma \ref{ExpectationRandom} completes the proof.
\end{proof}

To complete the proof of Theorem \ref{Omega} we need to derive lower bounds for the moments of $\sum_{n\leq x} \X(n)$. We establish the following proposition. 

\begin{pro}\label{BoundMomentsRandom}
Let $\ell\geq 2$ be an integer. Then, for all real numbers $2\leq y\leq x$ we have
$$
\ex\left(\left|\sum_{n\leq x}\X(n)\right|^{2\ell}\right) \geq \Psi(x,y;\tau)^{2\ell} \exp\left(-10\frac{y\log\log x}{\log y}+O\left(\frac{\ell}{\log x}\right)\right).
$$
\end{pro}
\begin{proof}
First, by Corollary \ref{CompareArithmetic} with $g(n)=1$ and $h(n)$ being the characteristic function of the $y$-friable numbers, we get
\begin{equation}\label{Smooth}
\ex\left(\left|\sum_{n\leq x}\X(n)\right|^{2\ell}\right)\geq \ex\left(\left|\sum_{\substack{n\leq x\\ P(n)\leq y}}\X(n)\right|^{2\ell}\right).
\end{equation}
For a prime $p$, write 
$$g^{\natural}_p= \left(\begin{matrix} e^{i \theta_p}&0\\0&e^{-i \theta_p}\end{matrix}\right)^{\natural},$$ where $\theta_p$ is a random variable taking values in $[0, \pi]$ and distributed according to the Sato-Tate distribution $d\mu_{st}(t):=\frac{2}{\pi}\sin^2(t)dt$ . Let $\mathcal{A}(\X)$ be the event corresponding to 
$$ |\theta_p|\leq \frac{1}{(\log x)^2}, \text{ for all primes } p\leq y.$$
By the independence of the $\theta_p$ for different primes $p$, we deduce that the probability of $\mathcal{A}(\X)$ is
$$ \pr(\mathcal{A}(\X)) =\left(\frac{2}{\pi}\int_0^{(\log x)^{-2}} \sin^2t dt\right)^{\pi(y)}\geq \left(\frac{c}{(\log x)^6}\right)^{\pi(y)}\gg \exp\left(-10 \frac{y \log\log x}{\log y}\right),$$
for some positive constant $c$. 
On the other hand, one can see that for any prime $p\leq y$ and all outcomes in $\mathcal{A}(\X)$, we have
$$ \X(p^a)=  \text{tr}(\text{Sym}^{a} g^{\natural}_{p})= \frac{\sin((a+1)\theta_p)}{\sin \theta_p}= (a+1)\left(1+O\left(a^2\theta_p^2\right)\right)=\tau(p^a) \left(1+O\left(\frac{a^2}{(\log x)^4}\right)\right).$$
Therefore, if $n\leq x$ and $P(n)\leq y$ then for all outcomes in $\mathcal{A}(\X)$ we have
$$
 \X(n)=\tau(n)  \left(1+O\left(\frac{\omega(n)}{(\log x)^2}\right)\right)= \tau(n)  \left(1+O\left(\frac{1}{\log x}\right)\right),
$$
where $\omega(n)$ is the number of distinct prime factors of $n$, which satisfies $\omega(n)\ll \log x$. Thus, we deduce that
\begin{align*}
\ex\left(\left|\sum_{\substack{n\leq x\\ P(n)\leq y}}\X(n)\right|^{2\ell}\right) &\geq \left(\sum_{\substack{n\leq x\\ P(n)\leq y}}\tau(n)  \left(1+O\left(\frac{1}{\log x}\right)\right)\right)^{2\ell} \pr(\mathcal{A}(\X))\\
&\gg \Psi(x,y;\tau)^{2\ell} \exp\left(-10\frac{y\log\log x}{\log y}+O\left(\frac{\ell}{\log x}\right)\right),
\end{align*}
as desired.
\end{proof}
We finish this section by proving Theorem \ref{Omega}, and deducing Corollary \ref{Omega2}.
\begin{proof}[Proof of Theorem \ref{Omega}]
Let $\ell= [\log k/(6 \log x)]$. Then, it follows from Proposition \ref{Moments} and Proposition \ref{BoundMomentsRandom} that 
\begin{align*}
\frac{1}{|\h|_h}\sum_{f\in \h} \omega_f \left|\sum_{n\leq x} \lambda_f(n)\right|^{2\ell} 
&\geq 
\Psi(x,y;\tau)^{2\ell} \exp\left(-10\frac{y\log\log x}{\log y}+O\left(\frac{\ell}{\log x}\right)\right)+O\left(k^{-1/3}\right)\\
& \geq \Psi(x,y;\tau)^{2\ell} \exp\left(-15\frac{\log k\log\log\log k}{(\log\log k)^2}\right). 
\end{align*}
Therefore, in view of \eqref{Harmonic1} and \eqref{Harmonic2} we obtain
\begin{equation}\label{LowerNatural}\sum_{f\in \h}  \left|\sum_{n\leq x} \lambda_f(n)\right|^{2\ell} \geq 
 \Psi(x,y;\tau)^{2\ell} \cdot k\exp\left(-20\frac{\log k\log\log\log k}{(\log\log k)^2}\right).
 \end{equation}
Let $\mathcal{B}$ be the set of Hecke cusp forms $f\in \h$ such that
$$
\left|\sum_{n\leq x} \lambda_f(n)\right| \geq \Psi(x,y;\tau) \left(1-\frac{1}{\sqrt{\log\log k}}\right).
$$
Since $|\h|\asymp k$ we obtain
$$ 
\sum_{f\in \h\setminus \mathcal{B}}\left|\sum_{n\leq x} \lambda_f(n)\right|^{2\ell}
\leq \Psi(x,y;\tau)^{2\ell} \cdot  k \exp\left(-\frac{\log k}{10A(\log\log k)^{3/2}}\right).
$$
Combining this bound with \eqref{LowerNatural} we get
\begin{equation}\label{LowerNatural2}
\sum_{f\in \mathcal{B}}  \left|\sum_{n\leq x} \lambda_f(n)\right|^{2\ell} \gg 
 \Psi(x,y;\tau)^{2\ell} \cdot k\exp\left(-20\frac{\log k\log\log\log k}{(\log\log k)^2}\right).
\end{equation}
On the other hand, we have 
$$
\sum_{f\in \mathcal{B}}  \left|\sum_{n\leq x} \lambda_f(n)\right|^{2\ell}\leq |\mathcal{B}| \left(\sum_{n\leq x}\tau(n)\right)^{2\ell}.
$$
Moreover, by \eqref{Bruijn} we have 
\begin{equation}\label{LowerBoundPSI}
\Psi(x,y;\tau)\geq \Psi(x, x^{1/(2A)}; \tau)\gg_A \sum_{n\leq x}\tau(n).
\end{equation}
Hence, we derive from \eqref{LowerNatural2} that
$$ |\mathcal{B}| \geq k 
 \exp\left(-20\frac{\log k\log\log\log k}{(\log\log k)^2}+O_A\left(\frac{\log k}{\log\log k}\right)\right),$$
 which completes the proof.
\end{proof}
\begin{proof}[Proof of Corollary \ref{Omega2}]
The result follows from Theorem \ref{Omega} together with Eq. \eqref{LowerBoundPSI}. 
\end{proof}

\section{Cancellations under GRH: proofs of Theorem \ref{Conditional} and Corollary \ref{Conditional2}}
Let $f\in \h$. For $\re(s)>1$ we have 
$$\log L(s, f)= \sum_{n=2}^{\infty} \frac{\Lambda(n)b_f(n)}{n^{s}\log n}, $$
where $b_f(n)=(e^{i\theta_f(p)})^a+(e^{-i\theta_f(p)})^a$ if $n=p^a$ for some prime $p$, and equals $0$ otherwise. For $y\geq 1$ we define
$$ 
L_y(s, f)= \sum_{P(n)\leq y}\frac{\lambda_f(n)}{n^s}= \prod_{p\leq y}\left(1-\frac{e^{i\theta_f(p)}}{p^s}\right)^{-1}\left(1-\frac{e^{-i\theta_f(p)}}{p^s}\right)^{-1}.
$$
In order to approximate $S_f(x)$ by $\Psi(x, y; \lambda_f)$, we shall prove that conditionally on GRH, $\log L(s, f)$ is very well approximated by $\log L_y(s, f)$ for $\re(s)\geq 1$. This will be the key ingredient in the proof of Theorem \ref{Conditional}.
\begin{lem}\label{ApproxFriable}
 Let $f\in \h$ and assume GRH for $L(s, f)$. Let $2\leq y\leq k$, and $s=\sigma+it$ with $\sigma\geq  1$ and $|t|\leq 2k$. Then, we have
$$ \log L(s, f)-\log L_y(s, f)\ll \frac{(\log y)^2\log k}{\sqrt{y}}.$$
\end{lem}
To prove this result we need the following standard bound.
\begin{lem}\label{Borel}
Let $f\in \h$. Let  $s=\sigma+it$ with $1/2<\sigma\leq 3/2$ and $|t|\leq 3k$. Let $1/2\leq \sigma_0<\sigma$, and suppose that there are no zeros of $L(z, f)$ inside the rectangle $\{z: \sigma_0\leq \re(z)\leq 1, |\im(z)-t|\leq 3\}$. Then, we have
$$\log L(s, f)\ll \frac{\log k}{\sigma-\sigma_0}.$$
\end{lem} 
\begin{proof}
Consider the circles with centre $2+it$ and radii $r=2-\sigma$ and $R=2-\sigma_0$, so that the smaller circle passes through $s$. By our assumption, $\log L(z, f)$ is analytic inside the larger circle. For a point $z$ on the larger circle, it follows from the standard convexity bound for $L(s, f)$ that $\re \log L(z, f)\ll \log k$. Finally, using the Borel-Caratheodory theorem we obtain
$$ 
\log L(s, f)\leq \frac{2r}{R-r}\max_{|z-2-it|=R}\re \log L(z, f)+ \frac{R+r}{R-r}|\log L(2+it, f)|\ll \frac{\log k}{\sigma-\sigma_0}.
$$
\end{proof}

\begin{proof}[Proof of Lemma \ref{ApproxFriable}]
Let $c_1=1-\sigma+1/\log y$.  Then it follows from Perron's formula (see \cite{Da}) that
\begin{align*}
\frac{1}{2\pi i} \int_{c_1-iy}^{c_1+iy} \log L(s+z, f) \frac{y^z}{z} dz &= \sum_{n\leq y} \frac{\Lambda(n)b_f(n)}{n^s\log n}+ O\left(y^{c_1}\sum_{n=1}^{\infty} \frac{1}{n^{\sigma+c_1}}\min\left(1, \frac{1}{y\log|y/n|}\right)\right)\\
&= \sum_{n\leq y} \frac{\Lambda(n)b_f(n)}{n^s\log n}+ O\left(y^{-\sigma}\log y\right),
\end{align*}
by a standard estimation of the error term. We now move the contour to the line $\re(s)=c_2$ where $c_2=1/2-\sigma+1/\log y$. By our assumption, we only encounter a simple pole at $z=0$ that leaves a residue of $\log L(s, f)$. Furthermore, it follows from Lemma \ref{Borel} with $\sigma_0=1/2$ that
$$ \log L(s+z, f)\ll \log k \log y,$$
uniformly for $z$ with $\re(z)\geq c_2$ and $|\im(z)|\leq y$. 
Therefore, we deduce that
$$ \frac{1}{2\pi i} \int_{c_1-iy}^{c_1+iy} \log L(s+z, f) \frac{y^z}{z} dz= \log L(s, f)+\mathcal{E}, $$
where 
\begin{align*}
\mathcal{E}&= \frac{1}{2\pi i} \left(\int_{c_1-iy}^{c_2-iy}+\int_{c_2-iy}^{c_2+iy}+ \int_{c_2+iy}^{c_1+iy}\right)\log L(s+z, f) \frac{y^z}{z} dz\ll \frac{(\log y)^2\log k}{y^{1/2-\sigma}}.
\end{align*}
The result follows upon noting that 
$$ 
\log L_y(s, f)- \sum_{n\leq y} \frac{\Lambda(n)b_f(n)}{n^s\log n}= 
\sum_{\substack{p\leq y\\ p^a>y}} \frac{(e^{i\theta_f(p)})^a+(e^{-i\theta_f(p)})^a}{ap^{as}}\ll \sum_{\substack{p\leq y\\ p^a>y}} \frac{1}{a p^{2a}}\ll \frac{1}{\sqrt{y}}. 
$$
\end{proof}
We now prove Theorem \ref{Conditional}. 

\begin{proof}[Proof of Theorem \ref{Conditional}]
Without loss of generality assume that $x\in \mathbb{Z}+1/2$. Let $c=1+1/\log x$. By Perron's formula  together with \eqref{Deligne} we have
$$
 \sum_{n\leq x} \lambda_f(n)
=\frac{1}{2\pi i}\int_{c-ix}^{c+ix} L(s, f) \frac{x^s}{s}ds+ O\left(\frac{1}{x} \sum_{n=1}^{\infty} \frac{x^c }{n^c}\frac{\tau(n)}{|\log (x/n)|}\right).
$$
The error term above is 
\begin{align*}
&\ll \sum_{n=1}^{\infty} \frac{\tau(n)}{n^{1+1/\log x}}+ \sum_{x/2<n<2x} \frac{\tau(n)}{n|\log(x/n)|}\\
&\ll_{\varepsilon} \zeta(1+1/\log x)^2+ x^{\varepsilon/2}\sum_{r\leq x}\frac{1}{r} \ll_{\varepsilon}  x^{\varepsilon}.
\end{align*}
Similarly, we have 
$$ \Psi(x, y; \lambda_f)= \frac{1}{2\pi i}\int_{c-ix}^{c+ix} L_y(s, f) \frac{x^s}{s}ds+ O_{\varepsilon} (x^{\varepsilon}).$$
Define 
$$
R_y(s, f):= \frac{L(s, f)}{L_y(s, f)}.
$$
Then, combining the above estimates we get
\begin{align*}
 \sum_{n\leq x} \lambda_f(n)- \Psi(x, y; \lambda_f)
 &= \frac{1}{2\pi i}\int_{c-ix}^{c+ix}\left(L(s, f)- L_y(s, f)\right) \frac{x^s}{s}ds+ O_{\varepsilon} (x^{\varepsilon})\\
 &= \frac{1}{2\pi i}\int_{c-ix}^{c+ix} L_y(s, f) \big(\exp(\log R_y(s, f))-1\big)\frac{x^s}{s}ds+ O_{\varepsilon} (x^{\varepsilon}).\\
\end{align*}
Moreover, using Lemma \ref{ApproxFriable} we obtain
$$
 \exp(\log R_y(s, f))= 1+O\left(\frac{(\log y)^2\log q}{\sqrt{y}}\right),
$$
for all $s$ with $\re(s)=c$ and $|\im(s)|\leq x$.
Furthermore, note that for $\re(s)=c$ we have 
$$ 
L_y(s, f)= \exp\left(\sum_{p\leq y} \frac{\lambda_f(p)}{p^s}+O(1)\right) \ll (\log y)^2,
$$
since $\lambda_f(p)\leq 2$. 
Combining these estimates, we deduce that
$$
 \frac{1}{2\pi i}\int_{c-ix}^{c+ix} L_y(s, f) \big(\exp(\log R_y(s, f))-1\big)\frac{x^s}{s}ds\ll \frac{x(\log x)(\log y)^4\log q}{\sqrt{y}},
$$ 
which completes the proof.

\end{proof}

In order to deduce Corollary \ref{Conditional2}, we need to prove the bound \eqref{DecayFriable},  which shows that $\Psi(x, y; \tau)=o(x\log x)$ when $u=\log x/\log y\to \infty$. 

\begin{lem}\label{Friable}
Let $10\leq y\leq x$ be real numbers. Then we have
$$\Psi(x, y; \tau)\ll e^{-u/2} x\log x .$$
\end{lem}
\begin{proof}
Let $\beta=2/(3\log y)$. Then, observe that
$$
 \Psi(x, y; \tau)\leq \sum_{n\leq x^{3/4}} \tau(n)+ x^{-3\beta/4}\sum_{\substack{x^{3/4}\leq n\leq x\\ P(n)\leq y}} n^{\beta}\tau(n)\ll x^{3/4}\log x+ e^{-u/2}\sum_{\substack{n\leq x\\ P(n)\leq y}} n^{\beta}\tau(n).$$
Let 
$$
g(n)=\begin{cases} n^{\beta}\tau(n) & \text{ if } P(n)\leq y,\\
0 & \text{ otherwise}.\end{cases}
$$
 Then $g$ is multiplicative, and for all primes $p\leq y$ we have $g(p^a)=(a+1)p^{a\beta}\ll (1.9)^a$. Therefore, by Corollary 3.5.1 of \cite{Te} we obtain
$$ \sum_{\substack{n\leq x\\ P(n)\leq y}} n^{\beta}\tau(n) \ll \frac{x}{\log x} \prod_{p\leq y} \left(\sum_{a=0}^{\infty} \frac{g(p^a)}{p^a}\right)\ll \frac{x}{\log x} \prod_{p\leq y} \left(1+\frac{2p^{\beta}}{p}\right).
$$
The result follows upon noting that $x^{3/4}\ll x e^{-u/2}$ for $y\geq 10$ and
$$ 
\prod_{p\leq y} \left(1+\frac{2p^{\beta}}{p}\right) \ll \exp\left(2\sum_{p\leq y} \frac{1+O(\beta \log p)}{p}\right)\ll (\log y)^2.
$$
\end{proof}

\begin{proof}[Proof of Corollary \ref{Conditional2}]
The result holds trivially for $x>k$ by \eqref{WeakRangeGRH}, so we may assume that $x\leq k$. Then, using Theorem \ref{Conditional} with $y=(\log k)^3$, together with Lemma \ref{Friable} and our assumption on $x$ completes the proof.
\end{proof}

\section*{Acknowledgements} 

I would like to thank Emmanuel Kowalski for useful comments concerning the probabilistic random model for the Hecke eigenvalues in Section 3.

\end{document}